\documentclass[12pt]{amsart}

\usepackage[all]{xy}

\newtheorem{theorem}{Theorem}[section]
\newtheorem{proposition}[theorem]{Proposition}

\newtheorem{corollary}[theorem]{Corollary}
\newtheorem{remark}[theorem]{Remark}

\numberwithin{equation}{section}

\newcommand{\ZZ}{\mathbb{Z}}

\begin{document}

\baselineskip=16pt

\title[Equivariant bundles and logarithmic connections on toric
varieties]{Equivariant vector bundles and logarithmic connections
on toric varieties}

\author[I. Biswas]{Indranil Biswas}
\address{School of Mathematics, Tata Institute of Fundamental
Research, Homi Bhabha Road, Bombay 400005, India}
\email{indranil@math.tifr.res.in}

\author[V. Mu\~{n}oz]{Vicente Mu\~{n}oz}
\address{Facultad de Ciencias Matem\'aticas, Universidad
Complutense de Madrid, Plaza Ciencias 3, 28040 Madrid, Spain}
\email{vicente.munoz@mat.ucm.es}

\author[J. S\'anchez]{Jonathan S\'anchez}
\address{Facultad de Ciencias Matem\'aticas, Universidad
Complutense de Madrid, Plaza Ciencias 3, 28040 Madrid, Spain}
\email{jnsanchez@mat.ucm.es}



\thanks{Partially supported by Spanish MICINN grant MTM2010-17389.}

\date{30 December, 2013}

\begin{abstract}
Let $X$ be a smooth complete complex toric variety such that the boundary
is a simple normal crossing divisor, and let $E$ be a holomorphic
vector bundle on $X$.
We prove that the following three statements are equivalent:
\begin{itemize}
\item The holomorphic vector bundle $E$ admits an equivariant structure.

\item The holomorphic vector bundle $E$ admits an integrable logarithmic
connection singular over $D$.

\item The holomorphic vector bundle $E$ admits a logarithmic connection
singular over $D$.
\end{itemize}
We show that an equivariant vector bundle on $X$ has a tautological
integrable logarithmic connection singular over $D$.
This is used in computing the Chern classes of the equivariant vector
bundles on $X$. We also prove a version of the above result for holomorphic
vector bundles on log parallelizable $G$-pairs $(X,
D)$, where $G$ is a simply connected complex affine algebraic group.
\end{abstract}

\maketitle

\section{Introduction} \label{sec:1}

Let $X$ be a smooth complete variety over $\mathbb C$. A reduced effective
divisor $D$ on $X$ is called \textit{locally simple normal crossing} if
around every point $x\,\in\, D$ there are holomorphic coordinate functions
$\{z_1\, ,\ldots\, , z_d\}$ on some analytic open subset $U_x\, \subset\, X$
containing $x$ such that $x=(0,\ldots, 0)$ and
$$
D\cap U_x\, =\, \{(z_1\, ,\cdots\, , z_d)\, \mid\, \, \prod_{i=1}^{d'} z_i\,=\, 0\}\cap U_x
$$
for some $d'\,\leq\, d$. A locally simple normal crossing divisor is called
a \textit{simple normal crossing} divisor if each irreducible component of
it is smooth.

\medskip

Now take $X$ to be a smooth complete complex toric variety. Let
$\mathcal T$ denote
the complex torus that acts on $X$ defining its toric structure.
The dimension of $X$ will be denoted by $d$.
Let ${\mathcal U}\, \subset\,X$ be the nonempty
Zariski open subset on which the action of $\mathcal T$ is free.

Then 
\begin{itemize}
\item the complement 
$$
D\, :=\, X\setminus \mathcal U
$$ 
is a simple normal crossing divisor, and

\item the logarithmic tangent bundle $TX(-\log D)$ is holomorphically
trivial.
\end{itemize}
See \cite[p. 473--474, Main Theorem]{Wi}.

An equivariant holomorphic vector bundle on $X$ is a holomorphic vector bundle
$E\, \longrightarrow\, X$ equipped with a holomorphic lift of the action
of $\mathcal T$.

In Section \ref{sec:2} we prove our first main result (see Theorem
\ref{thm1-1}):

\begin{theorem}\label{thm1}
Let $E$ be a holomorphic vector bundle over the toric variety $X$. The following
three statements are equivalent:
\begin{enumerate}
\item The holomorphic vector bundle $E$ admits an equivariant structure.

\item The holomorphic vector bundle $E$ admits an integrable logarithmic
connection singular over $D$.

\item The holomorphic vector bundle $E$ admits a logarithmic connection
singular over $D$.
\end{enumerate}
\end{theorem}

In fact, we show that an equivariant holomorphic vector bundle
on $X$ has a tautological
integrable logarithmic connection singular over $D$ (see Proposition \ref{prop2}).

Using the above mentioned tautological integrable logarithmic connection on an
equivariant holomorphic vector bundle $E$ on the toric variety $X$, we compute
the Chern classes of $E$ (see Section \ref{sec:5}).
Let $N$ be the lattice associated to the toric variety, and let
$M\,=\,\mathrm{Hom}_\ZZ(N,\ZZ)$ be the
dual lattice. The toric variety is defined by a fan $\Delta$. For each
one dimensional face $n\,\in\, \Delta_1$, let $X_n$ denote the corresponding
 boundary divisor. For $m\,\in\, M$, we have integers $d_{nm}\,\geq\, 0$
defined in \eqref{dmn}. Then the Chern classes of $E$ are given by 
 \begin{equation*}
	 c_k(E) = \sum_{\substack{m_1,\ldots,m_k\in M \\ m_i\neq m_j}}
	 \prod_{i=1}^k \left(\sum_{n\in \Delta_1}  d_{nm_i} \langle
	 m_i,n\rangle[X_n]\right)
 \end{equation*}
(see \eqref{eq:Chern_class}).

\medskip

Theorem \ref{thm1} holds for a larger family of varieties. 
Let $G$ be a complex affine algebraic group. Let $X$ be a smooth
complete complex variety on which
$G$ acts. Let $D\,\subset\, X$ be a simple normal crossing divisor
preserved by the action of $G$ on $X$.
This $(X\, ,D)$ is called a $G$--\textit{pair}.

Let $\mathfrak g$ be the Lie algebra of $G$.
For a $G$--pair $(X\, ,D)$, we have a homomorphism of coherent sheaves
\begin{equation*}
\mathrm{op}_{X,D}\,:\,X\times \mathfrak{g}\,
\longrightarrow\, {T}X(-\log D)
\end{equation*}
induced by the action of $G$. The $G$--pair $(X\, ,D)$ is called
\textit{log parallelizable} if the above homomorphism
$\mathrm{op}_{X,D}$ is an isomorphism. See \cite{Br} for properties
of log parallelizable $G$--pairs.

An equivariant vector bundle on a $G$--pair $(X\, ,D)$ is a holomorphic
vector bundle on $X$ equipped with a lift of the action of $G$ on $X$.

In Section \ref{sec:3} we prove the following (see Theorem
\ref{thm:prop-G-1}):

\begin{theorem}\label{thm:prop-G}
Assume that $G$ is simply connected. Let $(X\, ,D)$ be a log parallelizable
$G$--pair. Let $E$ be a holomorphic vector bundle on $X$.
The following two statements are equivalent:
\begin{enumerate}
\item The holomorphic vector bundle $E$ admits an equivariant structure.

\item The holomorphic vector bundle $E$ admits an integrable logarithmic
connection singular over $D$.
\end{enumerate}
\end{theorem}

In Section \ref{ex} we give an example showing that the assumption
in Theorem \ref{thm:prop-G} that $G$ is simply connected is necessary.

If $G$ is semisimple, then the two statements in Theorem \ref{thm:prop-G}
is equivalent to the third statement of Theorem \ref{thm1}. More
precisely, we prove the following (see Corollary \ref{cor-s}):

\begin{proposition}\label{prop-s2}
Assume that $G$ is semisimple and simply connected. Let $E$ be a holomorphic
vector bundle on $X$. Then the two statements
in Theorem \ref{thm:prop-G} is equivalent to the following: The vector
bundle $E$ admits a logarithmic connection singular over $D$.
\end{proposition}

A key step in the proof of Theorem \ref{thm1} is a theorem of Klyachko which
says that a holomorphic vector bundle $E$ on a toric variety $X$ admits
an equivariant structure if and only if all the pullbacks of $E$ by the
elements of the torus acting on $X$ are holomorphically isomorphic to $E$.
The above mentioned example of Section \ref{ex} also shows that this criterion fails for
holomorphic vector bundles on a general log parallelizable $G$--pair.

\medskip

We next consider equivariant vector bundles over a
complete toric variety $X$ defined over
an algebraically closed field of arbitrary characteristic.
An equivariant vector bundle $E$ over $X$ is nef (respectively, trivial)
if and only if its restriction to any invariant curve
$C\,\subset\, X$ is nef (respectively, trivial) \cite{HMP}.
In Section \ref{sec:4}, we see that such an explicit result cannot be expected
for \emph{semistable} equivariant bundles. In fact, if an equivariant
vector bundle on $X$ is semistable when restricted to any invariant
curve, then it is decomposable into
a direct sum of copies of one line bundle (Proposition \ref{prop1}).

\medskip
\noindent \textbf{Acknowledgements:}\,
We are grateful to the referee for useful comments on a previous version of
the paper. The first-named author wishes
to thank ICMAT for its hospitality while the work was carried out.

\section{Logarithmic connection on equivariant bundles} \label{sec:2}

In this section, the base field is taken to be the complex numbers.

Let $X$ to be a smooth complete toric variety of dimension $d$. Let $\mathcal T$
denote
the complex torus that acts on $X$ defining its toric structure.
Let ${\mathcal U}\, \subset\,X$ be the nonempty
Zariski open subset on which the action of $\mathcal T$ is free. Let
$$D\, :=\, X\setminus \mathcal U\,$$
be the complement.
Recall that $D$ is a simple normal crossing divisor.

\begin{proposition}\label{prop2}
Let $E$ be an equivariant vector bundle on $X$. Then $E$ has a natural
integrable logarithmic connection singular over $D$.
\end{proposition}

\begin{proof}
The Lie algebra of the torus $\mathcal T$ acting on $X$
will be denoted by $\mathfrak t$. Let
$$
V\, :=\, X\times {\mathfrak t}\, \longrightarrow\, X
$$
be the trivial holomorphic vector bundle with fiber $\mathfrak t$. The action
of $\mathcal T$ on $X$ defines an ${\mathcal O}_X$--linear homomorphism
\begin{equation}\label{e4}
V\, \longrightarrow\, TX\, .
\end{equation}
We recall that $TX(-\log D)$ is the subsheaf of $TX$ generated by all locally
defined holomorphic
vector fields that take the subsheaf ${\mathcal O}_X(-D)\,\subset\,
{\mathcal O}_X$ to ${\mathcal O}_X(-D)$. It is easy to see that the image of
the homomorphism in \eqref{e4} is contained in the subsheaf $TX(-\log D)\,
\subset\, TX$. Indeed, this follows immediately from
the fact that the action of $\mathcal T$ on $X$ preserves $D$.

Let
\begin{equation}\label{e3}
\beta\, :\, V\, \longrightarrow\, TX(-\log D)
\end{equation}
be the homomorphism obtained above from the action of $\mathcal T$
on $X$. We noted earlier that $TX(-\log D)$ is holomorphically trivial
by \cite[p. 473--474, Main Theorem]{Wi}.

Any holomorphic homomorphism of holomorphically trivial vector bundles
on $X$ which is isomorphism over some point of $X$ can be shown
to be an isomorphism over entire $X$. To prove this, for any such homomorphism
$\psi\, :\, E_1\, \longrightarrow\, E_2$, consider the homomorphism
$$
\bigwedge\nolimits^r \psi\, :\, \bigwedge\nolimits^r E_1\, \longrightarrow\,
\bigwedge\nolimits^r E_2\, ,
$$
where $r$ is the rank of $E_1$ (same as the rank of $E_2$ because $\psi$
is an isomorphism over some point of $X$). Therefore, $\bigwedge^r \psi$
defines a nonzero holomorphic section of the holomorphic line bundle
$Hom(\bigwedge^r E_1\, ,\bigwedge^r E_2)\,=\,
{\mathcal O}_X$ (recall that both $E_1$ and $E_2$
are trivial). But any holomorphic function on $X$ non identically
zero does not vanish anywhere. Hence the homomorphism $\bigwedge^r \psi$
is nowhere vanishing. This implies that $\psi$ is an isomorphism over
$X$.

Since both $V$ and $TX(-\log D)$ are holomorphically trivial, and
the homomorphism $\beta$ in \eqref{e3} is an isomorphism over $\mathcal U$, we
conclude that $\beta$ is an isomorphism over $X$.
\medskip

Let $\text{At}(E)$ be the Atiyah bundle for $E$. We quickly recall
the construction of $\text{At}(E)$ from \cite{Atiyah-bundle}. Let
\begin{equation}\label{f}
\delta\, :\, E_{\rm GL}\, \longrightarrow\, X
\end{equation}
be the principal $\text{GL}(r, {\mathbb C})$--bundle associated to $E$, where
$r$ is the rank of $E$; so the fiber of $E_{\rm GL}$ over any point
$x\, \in\, X$ is the space is all linear isomorphisms from ${\mathbb C}^r$
to the fiber $E_x$. The action of $\text{GL}(r, {\mathbb C})$ on
$E_{\rm GL}$ produces an action of $\text{GL}(r, {\mathbb C})$ on
the direct image $\delta_*(TE_{\rm GL})$. Then
$$
\text{At}(E)\, :=\, (\delta_*(TE_{\rm GL}))^{\text{GL}(r,
{\mathbb C})}\, \subset\, \delta_*(TE_{\rm GL})
$$
is the invariant direct image.
We note that the Lie bracket operation of locally defined
holomorphic vector fields on
$E_{\rm GL}$ produces a Lie algebra structure on the sheaf of sections of
$\text{At}(E)$.

The Atiyah bundle $\text{At}(E)$ fits in a short
exact sequence of holomorphic vector bundles over $X$
 \begin{equation}\label{e9}
 0\,\longrightarrow\, End(E)\,\longrightarrow\, \text{At}(E)
 \,\stackrel{\overline{\eta}}{\longrightarrow}\, TX\,\longrightarrow\, 0
 \end{equation}
which is known as the \textit{Atiyah exact sequence} for $E$. The
Atiyah exact sequence produces the short exact sequence
\begin{equation}\label{e5}
 0\,\longrightarrow\, End(E)(-\log D) \,\longrightarrow\, \text{At}(E)(-\log D)
 \,\stackrel{\eta}{\longrightarrow}
\, TX(-\log D)\,\longrightarrow\, 0\, ,
\end{equation}
where
\begin{equation}\label{fd1}
\text{At}(E)(-\log D) \, :=\, \overline{\eta}^{-1}(TX(-\log D))\, .
\end{equation}
The homomorphism $\eta$ in \eqref{e5} is the restriction of $\overline\eta$
in \eqref{e9}. Note that
$$
\text{At}(E)(-\log D) \, =\, (\delta_*(TE_{\rm GL}(-\log \delta^{-1}(D)))
)^{\text{GL}(r, {\mathbb C})}\, ,
$$
where $\delta$ is the projection in \eqref{f}.
Therefore the Lie algebra structure on the sheaf of sections of $\text{At}(E)$
preserves the sheaf of sections of $\text{At}(E)(-\log D)$.

A logarithmic connection on $E$ singular over $D$ is, by definition,
a holomorphic splitting of the short exact sequence in \eqref{e5}.

The pulled back vector bundle $\delta^*E\,\longrightarrow\,E_{\rm GL}$
is canonically identified with the trivial vector bundle on $E_{\rm GL}$
with fiber ${\mathbb C}^r$ (recall that the points of $E_{\rm GL}$ are
isomorphisms of ${\mathbb C}^r$ with fibers of $E$). Therefore, we can
take the de Rham differential of the locally defined sections of
$\delta^*E$. Using it, a logarithmic connection $\nabla$ on $E$ singular over $D$
is identified with a holomorphic differential operator of order one
\begin{equation}\label{co-op}
\widetilde{\nabla}\, :\, E\,\longrightarrow\, E\otimes\Omega_X(\log D)
\,=\, E\otimes(TX(-\log D))^*
\end{equation}
satisfying the Leibniz identity which says that
$\widetilde{\nabla}(fs) \,=\, f\widetilde{\nabla}(s)+
s\otimes(df)$, where $s$ is any locally defined holomorphic section of $E$
and $f$ is any locally defined holomorphic function on $X$.

The action of $\mathcal T$ on $E$ produces an action of $\mathcal T$
on the total space of the principal bundle $E_{\rm GL}$ in \eqref{f}. Therefore,
we get a ${\mathcal O}_{E_{\rm GL}}$--linear homomorphism
$$
\zeta\, :\, E_{\rm GL}\times {\mathfrak t}\, \longrightarrow\, TE_{\rm GL}\, ,
$$
where $E_{\rm GL}\times {\mathfrak t}$ is the trivial holomorphic vector
bundle on $E_{\rm GL}$ with fiber ${\mathfrak t}$. The action of
$\text{GL}(r, {\mathbb C})$ on $E_{\rm GL}$ and the trivial action of
$\text{GL}(r, {\mathbb C})$ on $\mathfrak t$ together produce an
action of $\text{GL}(r, {\mathbb C})$ on $E_{\rm GL}\times {\mathfrak t}$.
The above homomorphism $\zeta$ is clearly $G$--equivariant. Therefore,
$\zeta$ descends to a ${\mathcal O}_X$--linear homomorphism
\begin{equation}\label{e6}
V\,=\, X\times {\mathfrak t}\, \longrightarrow\, \text{At}(E)\, .
\end{equation}
Since the image of the homomorphism in \eqref{e4} is contained in
$TX(-\log D)$, and the diagram
$$
\begin{matrix}
V & \longrightarrow & \text{At}(E)\\
\Vert && ~\Big\downarrow\overline{\eta}\\
V & \longrightarrow & TX
\end{matrix}
$$
is commutative (the horizontal maps are as in \eqref{e6} and \eqref{e4}
respectively), it follows that the homomorphism in \eqref{e6} factors
through a homomorphism
\begin{equation}\label{e7}
\gamma\, :\, V\, \longrightarrow\, \text{At}(E)(-\log D)\, \subset\,
\text{At}(E)\, ,
\end{equation}
where $\text{At}(E)(-\log D)$ is defined in \eqref{fd1}.

We showed above that the homomorphism $\beta$ in \eqref{e3} is an isomorphism.
Consider the composition
$$
\gamma\circ \beta^{-1}\, :\, TX(-\log D)\,\longrightarrow\, \text{At}(E)(-\log D)\, .
$$
Since the action of $\mathcal T$ on $E$ is a lift of the action of $\mathcal T$ on $X$,
it follows that
$$
\eta\circ (\gamma\circ \beta^{-1})\,=\, \text{Id}_{TX(-\log D)}\, ,
$$
where $\eta$ is the homomorphism in \eqref{e5}.
Therefore, the homomorphism $\gamma\circ \beta^{-1}$
defines a logarithmic connection on $E$ singular over $D$.

The homomorphism in \eqref{e4} takes the Lie algebra structure on the fibers
of $V$ (recall that the fibers of $V$ are the abelian Lie
algebra $\mathfrak t$) to the Lie algebra structure on the
sheaf of holomorphic vector
fields on $X$ defined by the Lie bracket operation. Similarly, the
homomorphism in \eqref{e6}
takes the Lie algebra structure on the fibers of $V$ to the Lie algebra
structure on the sheaf of sections of $\text{At}(E)$ (this Lie
algebra structure was explained earlier). Consequently, the
homomorphism
$\gamma\circ \beta^{-1}$ is compatible with the Lie algebra structure on the
sheaf of sections of $TX(-\log D)$ and $\text{At}(E)(-\log D)$ (recall
that the Lie algebra structure of the sheaf of sections of $\text{At}(E)$
preserves the sheaf of sections of $\text{At}(E)(-\log D)$). Therefore, the
logarithmic connection on $E$ defined by $\gamma\circ \beta^{-1}$ is integrable.
\end{proof}

\begin{theorem}\label{thm1-1}
Let $E$ be a holomorphic vector bundle over the toric variety $X$. The following
three statements are equivalent:
\begin{enumerate}
\item The vector bundle $E$ admits an equivariant structure.

\item The vector bundle $E$ admits an integrable logarithmic connection singular
over $D$.

\item The vector bundle $E$ admits a logarithmic connection singular
over $D$.
\end{enumerate}
\end{theorem}

\begin{proof}
{}From Proposition \ref{prop2} we know that the first statement implies
the second statement. The second statement immediately implies the
third statement.
So it is enough to prove that the third statement implies the first statement.

Let
\begin{equation}\label{rho}
\rho\, :\, {\mathcal T}\, \longrightarrow\, \text{Aut}(X)
\end{equation}
be the action of $\mathcal T$ on $X$.

Assume that $E$ admits a logarithmic connection singular over $D$.

It is
known that a holomorphic vector bundle $W$ on $X$ admits an equivariant
structure if for every $g\, \in\, \mathcal T$, the pulled back vector bundle
$\rho(g)^*W$ (see \eqref{rho}) is holomorphically isomorphic to $W$ \cite[p. 342,
Proposition 1.2.1]{Kl}.

We will show that the pulled back vector bundle $\rho(g)^*E$ is
holomorphically isomorphic to $E$ for all $g\, \in\, \mathcal T$.

Let ${\mathcal G}_E$ be the set of all pairs of the form $(f\, , \phi)$, where
$f\,\in\, \text{Aut}(X)$ and $\phi$ is a holomorphic automorphism of the
vector bundle $E$ over the automorphism $f$.
So $f$ and $\phi$ fit in a commutative diagram
$$
\begin{matrix}
E & \stackrel{\phi}{\longrightarrow} & E\\
\Big\downarrow && \Big\downarrow\\
X & \stackrel{f}{\longrightarrow} & X
\end{matrix}
$$
such that $\phi$ is fiberwise linear.
Note that $\phi$ defines a holomorphic isomorphism
of $E$ with the pullback $f^*E$. The natural composition operation makes
${\mathcal G}_E$ a group. The inverse of an element $(f\, , \phi)$ is
$(f^{-1}\, , \phi^{-1})$. This ${\mathcal G}_E$ is a complex Lie group with
Lie algebra $H^0(X, \, \text{At}(E))$. We recall that
the Lie algebra structure of
$H^0(X, \, \text{At}(E))$ is given by the Lie bracket of holomorphic
vector fields.

Let
\begin{equation}\label{eq1}
p\, :\, {\mathcal G}_E\, \longrightarrow\, \text{Aut}(X)
\end{equation}
be the homomorphism defined by $(f\, ,\phi)\,\longmapsto\, f$. The
corresponding homomorphism of Lie algebras
\begin{equation}\label{e8}
dp\, :\, \text{Lie}({\mathcal G}_E)\,=\, H^0(X, \, \text{At}(E))\,
\longrightarrow\, \text{Lie}({\rm Aut}(X))\,=\, H^0(X,\, TX)
\end{equation}
coincides with the homomorphism $H^0(X, \, \text{At}(E))\,
\longrightarrow\, H^0(X,\, TX)$ induced by $\overline\eta$ in \eqref{e9}.

Consider $\text{At}(E)(-\log D)$ defined in \eqref{fd1}. Let
$$
\nabla\,:\, TX(-\log D)\,\longrightarrow\, \text{At}(E)(-\log D)
$$
be a logarithmic connection on $E$ singular over $D$. Let
$$
\nabla'\,:\, H^0(X,\, TX(-\log D))\,\longrightarrow\, H^0(X,\, \text{At}(E)(-\log D))
$$
be the homomorphism induced by $\nabla$. We recall that $H^0(X,\, TX(-\log D))$
is identified with the Lie algebra $\mathfrak t$ using the action of
$\mathcal T$ on $X$ (the isomorphism $\beta$ in \eqref{e3} produces the
identification). Since $\eta\circ\nabla\,=\,
\text{Id}_{TX(-\log D)}$, where $\eta$ is the homomorphism in \eqref{e5}, it follows that
the composition with $dp$ in \eqref{e8}
$$
(dp)\circ \nabla'\, :\, H^0(X,\, TX(-\log D))\,\longrightarrow\, H^0(X,\, TX)
$$
actually coincides with the homomorphism given by the inclusion
$TX(-\log D)\, \hookrightarrow\, TX$; recall that $\text{At}(E)(-\log D)$ is a
subsheaf of $\text{At}(E)$, so the composition $(dp)\circ \nabla'$ is
well--defined.

Consequently, the image of $dp$ contains $H^0(X,\, TX(-\log D))$. Since the group
$\mathcal T$ is connected, and the homomorphism $\beta$ in \eqref{e3} is an
isomorphism, this implies that the image of the homomorphism $p$
in \eqref{eq1} contains $\rho({\mathcal T})$. Therefore, for each $g\, \in\,
\mathcal T$, there is a holomorphic isomorphism between $\rho(g)^*E$
and $E$. Now by the earlier mentioned criterion of Klyachko, the holomorphic
vector bundle $E$ admits an equivariant structure. This completes the proof.
\end{proof}

\subsection{Equivariant Krull--Schmidt decomposition} \label{sec:2.2}

An equivariant vector bundle on a smooth complete complex toric variety
$X$ is called \textit{decomposable} if it is a direct sum of two
equivariant vector bundles of positive rank. An equivariant vector bundle is
called \textit{indecomposable} if it is not decomposable.

For an equivariant vector bundle $E$ on $X$, let $\text{Aut}^{\mathcal T}(E)$
denote the group of all holomorphic
$\mathcal T$--equivariant automorphisms of the vector
bundle $E$. We note that $E$ is indecomposable if and only if the maximal
torus of $\text{Aut}^{\mathcal T}(E)$ is of dimension one (in other
words, the maximal torus is isomorphic to ${\mathbb C}^*$).

The next corollary follows from Theorem 4.1 of \cite{BP}. It also
follows from Klyachko's classification of toric vector bundles in \cite{Kl}.

\begin{corollary}\label{cor-KRS}
Let $E$ be an equivariant vector bundle over a smooth compete
complex toric variety $X$. Let
$$
E\, =\, \bigoplus_{i=1}^m E_i \,~ \text{  and  }~\,  E\, =\, \bigoplus_{i=1}^n F_i
$$
be two decompositions of $E$ into direct sum of indecomposable equivariant
vector bundles. Then $m\,=\, n$, and there is a permutation $\sigma$ of
$\{1\, ,\cdots\, , m\}$ such that the equivariant vector bundle $E_i$ is
isomorphic to the equivariant vector bundle $F_{\sigma(i)}$ for every $i\, \in\, \{1\, ,\cdots \, ,m\}$.
\end{corollary}

\section{Equivariant bundles on $G$--pairs} \label{sec:3}

Let $G$ be a connected complex affine algebraic group.

Let $X$ be a smooth complete complex variety equipped with an action
of $G$. Let $D\,\subset\, X$ be a simple normal crossing divisor
preserved by the action of $G$ on $X$.
This $(X\, ,D)$ is called a $G$--pair.

The Lie algebra of $G$ will be denoted by $\mathfrak g$. Let
$X\times \mathfrak{g}$ be the trivial
holomorphic vector bundle on $X$ with fiber $\mathfrak{g}$.
For a $G$--pair $(X\, ,D)$, we have a homomorphism of coherent sheaves
\begin{equation}\label{f1}
\mathrm{op}_{X,D}\,:\,X\times \mathfrak{g}\,\longrightarrow\, {T}X(-\log D)
\end{equation}
induced by the action of $G$.

The $G$--pair $(X\, ,D)$ is called \textit{log parallelizable} if the
homomorphism $\mathrm{op}_{X,D}$ in \eqref{f1} is an isomorphism.

An equivariant vector bundle on a $G$--pair $(X\, ,D)$ is a holomorphic
vector bundle on $X$ equipped with a lift of the action of $G$ on $X$.

\begin{proposition}\label{prop:cor-G}
Let $(X\, ,D)$ be a log parallelizable $G$--pair. Let $E$ be a
$G$--equivariant vector bundle on $X$. Then $E$ has a natural integrable
logarithmic connection singular over $D$.
\end{proposition}

\begin{proof}
The proof is identical to the proof of Proposition \ref{prop2}.
The isomorphism $\beta$ used in the proof of Proposition \ref{prop2}
is now replaced by the isomorphism $\mathrm{op}_{X,D}$ in
\eqref{f1}. Note that
the abelianness of the Lie algebra $\mathfrak t$ does not play
any r\^ole in the proof of Proposition \ref{prop2}.
\end{proof}

The following result is an analog of
Theorem \ref{thm1-1}.

\begin{theorem}\label{thm:prop-G-1}
Assume that the group $G$ is simply connected. Let $(X\, ,D)$ be a log
parallelizable $G$--pair. Let $E$ be a holomorphic vector bundle on $X$.
The following two statements are equivalent:
\begin{enumerate}
\item The vector bundle $E$ admits an equivariant structure.

\item The vector bundle $E$ admits an integrable logarithmic
connection singular over $D$.
\end{enumerate}
\end{theorem}

\begin{proof}
The second statement follows from the first statement by
Proposition \ref{prop:cor-G}.

To prove the converse, assume that the vector bundle $E$ admits an
integrable logarithmic connection singular over $D$.

As in the proof of Theorem \ref{thm1-1}, let ${\mathcal G}_E$ be the
complex group consisting of all pairs of the form $(f\, , \phi)$, where
$f\,\in\, \text{Aut}(X)$ and $\phi$ is a holomorphic automorphism of the
vector bundle $E$ over $f$. Consider
$$
dp\, :\, \text{Lie}({\mathcal G}_E)\,=\, H^0(X, \, \text{At}(E))\,
\longrightarrow\, \text{Lie}({\rm Aut}(X))\,=\, H^0(X,\, TX)
$$
constructed in \eqref{e8}. Let $\nabla\,:\, TX(-\log D)\,\longrightarrow\,
\text{At}(E)(-\log D)$ be a logarithmic connection on $E$ singular over
$D$. Let
 \begin{equation}\label{above}
\nabla'\,:\, H^0(X,\, TX(-\log D))\,\longrightarrow\,
H^0(X,\, \text{At}(E)(-\log D))
 \end{equation}
be the homomorphism induced by $\nabla$. As $\nabla$ is an integrable
connection, the map in \eqref{above} is a homomorphism of Lie algebras.

Using the isomorphism $\mathrm{op}_{X,D}$ in \eqref{f1}, the Lie algebra
$\mathfrak g$ of $G$ is
identified with the Lie algebra $H^0(X,\, TX(-\log D))$.
Therefore, we get a homomorphism
of Lie algebras
$$
{\mathfrak g}\,=\, H^0(X,\, TX(-\log D))\,
\stackrel{\nabla'}{\longrightarrow}\, H^0(X,\, \text{At}(E)(-\log D))
\, \hookrightarrow\, H^0(X,\, \text{At}(E))\,=\, \text{Lie}({\mathcal
G}_E)\, .
$$
Since $G$ is simply connected, this homomorphism of Lie algebras
${\mathfrak g}\,\longrightarrow\, \text{Lie}({\mathcal
G}_E)$ integrates into a homomorphism of Lie groups
$$
\beta\, :\, G\, \longrightarrow\, {\mathcal G}_E\, .
$$
This homomorphism $\beta$ evidently defines an equivariant structure
on $E$.
\end{proof}

The condition in Theorem \ref{thm:prop-G-1} that $G$ is simply connected
is necessary; see Section \ref{ex} below.

The proof that the third statement in Theorem \ref{thm1-1} implies the
first statement in Theorem \ref{thm1-1} breaks down for log parallelizable
$G$--pair due to the following reason: Proposition 1.2.1 of \cite{Kl},
which is crucially used in the proof of Theorem \ref{thm1-1}, is not
valid for a general log parallelizable $G$--pair; see Section \ref{ex} below
for such an example. However, if $G$ is semisimple,
then the analog of the third statement in Theorem \ref{thm1-1} implies the
analog of the second statement, as is shown in the following proposition.

\begin{proposition}\label{prop-s}
Assume that $G$ is semisimple. Let $(X\, ,D)$ be a log parallelizable $G$--pair. 
Let $E$ be a holomorphic vector bundle on $X$ admitting a logarithmic connection
singular over $D$. Then $E$ admits an integrable logarithmic connection
singular over $D$.
\end{proposition}

\begin{proof}
Consider the composition
$$
\text{At}(E)(-\log D)\,\stackrel{\eta}{\longrightarrow}\,
TX(-\log D)\,\stackrel{(\mathrm{op}_{X,D})^{-1}}{\longrightarrow}\,
X\times \mathfrak{g}\, ,
$$
where $\eta$ and $\mathrm{op}_{X,D}$ constructed in \eqref{e5} and \eqref{f1}
respectively. We will denote this composition by $\eta'$. Let
$$
{\widetilde\eta}\, :\, H^0(X,\, \text{At}(E)(-\log D))\,\longrightarrow\,
H^0(X,\,X\times \mathfrak{g})\,=\, \mathfrak g
$$
be the homomorphism given by $\eta'$.

Let $\nabla\, :\, TX(-\log D)\,\longrightarrow\, \text{At}(E)(-\log D)$
be a logarithmic connection on $E$ singular over $D$. We note that
$$
\eta'\circ\nabla\circ\mathrm{op}_{X,D}\,=\,\text{Id}_{X\times \mathfrak{g}}\, ,
$$
where $\mathrm{op}_{X,D}$ is the isomorphism in \eqref{f1}. This implies
that the above homomorphism ${\widetilde\eta}$ is surjective.

The surjective homomorphism ${\widetilde\eta}$ is a homomorphism of Lie
algebras, and $\mathfrak g$ is semisimple. Therefore, there is a homomorphism
of Lie algebras
$$
\theta\, :\, {\mathfrak g}\, \longrightarrow\,
H^0(X,\, \text{At}(E)(-\log D))
$$
such that ${\widetilde\eta}\circ\theta\,=\,\text{Id}_{\mathfrak g}$
(cf.\ \cite[p. 91, Corollaire 3]{Bou}).

The above homomorphism $\theta$ produces a homomorphism of ${\mathcal
O}_X$--coherent sheaves
$$
\nabla'\, :\, TX(-\log D)\,\longrightarrow\, \text{At}(E)(-\log D)
$$
because $TX(-\log D)$ is the trivial vector bundle with fiber $\mathfrak g$
(see \eqref{f1}). Clearly, $\nabla'$ is a logarithmic connection on $E$
singular over $D$. This logarithmic connection is integrable because
$\theta$ is a homomorphism of Lie algebras.
\end{proof}

Theorem \ref{thm:prop-G-1} and Proposition \ref{prop-s} together have
the following corollary:

\begin{corollary}\label{cor-s}
Assume that $G$ is semisimple and simply connected.
Let $(X\, ,D)$ be a log parallelizable $G$--pair. Let $E$ be a holomorphic
vector bundle on $X$. Then the following three statements are equivalent:
\begin{enumerate}
\item The holomorphic vector bundle $E$ admits an equivariant structure.

\item The holomorphic vector bundle $E$ admits an integrable logarithmic
connection singular over $D$.

\item The holomorphic vector bundle $E$ admits a logarithmic connection
singular over $D$.
\end{enumerate}
\end{corollary}

\subsection{An example}\label{ex}

Take $G\,=\, \text{PGL}(2, {\mathbb C})$. Let ${\mathbb P}(M(2,{\mathbb C}))$
be the projective space parametrizing
lines in the $2\times 2$ complex matrices. Let
$$
f\, :\, \text{PGL}(2, {\mathbb C})\, \longrightarrow\,
{\mathbb P}(M(2,{\mathbb C}))\,=\, {\mathbb C}{\mathbb P}^3
$$
be the natural embedding defined by the inclusion map of
$\text{GL}(2, {\mathbb C})$ in the vector space $M(2,{\mathbb C})$.
The right and left translation
actions of $\text{GL}(2, {\mathbb C})$ on
$M(2,{\mathbb C})$ produce respectively right and left actions of
$\text{PGL}(2,
{\mathbb C})$ on ${\mathbb P}(M(2,{\mathbb C}))$. The above embedding $f$
is the wonderful compactification of $\text{PGL}(2, {\mathbb C})$
(see \cite{DP1}, \cite{DP2}), but we do not use this.

Let
$$
D\, :=\, {\mathbb P}(M(2,{\mathbb C})) \setminus f(\text{PGL}(2,
{\mathbb C}))\, \subset\, {\mathbb P}(M(2,{\mathbb C}))\,:=\, X
$$
be the boundary divisor. We will show that the vector bundle
$TX(-\log D)$ is holomorphically trivial.

Note that $D$ is a smooth hypersurface of degree two.
Let $\iota\, :\, D\, \hookrightarrow\, X$ be the inclusion map.
Consider the short exact sequence of coherent sheaves on $X$
$$
0\, \longrightarrow\, TX(-\log D)\, \longrightarrow\, TX
\, \longrightarrow\, \iota_* N_D \,=\, {\mathcal O}_X(D)\vert_D
\, \longrightarrow\, 0\, ,
$$
where $N_D$ is the normal bundle of $D$; the above isomorphism
$$\iota_* N_D \,=\, {\mathcal O}_X(D)\vert_D$$ is the Poincar\'e adjunction
formula. Since the degree of $D$ is
two, we have
$$
\text{degree}({\mathcal O}_X(D)\vert_D)\,=\, 4\, .
$$
Hence from the
above short exact sequence it follows that
$$
\text{degree}(TX(-\log D))\,=\, 0\, .
$$

Let $V\,=\, X\times {\mathfrak{s}\mathfrak{l}} (2,{\mathbb C})$ be the trivial
holomorphic vector bundle over $X$ with fiber
${\mathfrak{s}\mathfrak{l}} (2,{\mathbb C})\,=\,
\text{Lie}(\text{PGL}(2, {\mathbb C}))$.
The action of $\text{PGL}(2, {\mathbb C})$ on $X$ defines a homomorphism
$$
\beta\, :\, V\, \longrightarrow\, TX(-\log D)\, .
$$
Let
$$
\bigwedge\nolimits^2 \beta\, :\, \bigwedge\nolimits^2 V\,=\, {\mathcal O}_X
\, \longrightarrow\, \bigwedge\nolimits^2 (TX(-\log D)\,=\, {\mathcal O}_X
$$
be the homomorphism induced by $\beta$. Since the homomorphism
$\bigwedge^2 \beta$ is an isomorphism over the complement of $D$, it follows
that $\bigwedge^2 \beta$ is an isomorphism over $X$ (a somewhere nonzero
holomorphic function on $X$ is nowhere vanishing). Therefore, $\beta$
is an isomorphism over $X$.

Consequently, $(X\, ,D)$ is a log parallelizable
$\text{PGL}(2, {\mathbb C})$--pair.

Let $L\,:=\, {\mathcal O}_{{\mathbb P}(M(2,{\mathbb C}))}(-1)$
be the tautological line bundle on $X$. We will show that $L$ admits
an integrable logarithmic connection singular over $D$. For this,
note that the line bundle $L^{\otimes 2}\,=\, {\mathcal O}_X(-D)$
has a tautological integrable logarithmic connection, singular over $D$,
given by the de Rham differential $\alpha\, \longmapsto\,
d\alpha$. For any holomorphic line bundle $\xi$, and any nonzero
integer $m$, a logarithmic connection on $\xi$ induces a logarithmic
connection on $\xi^{\otimes m}$; moreover, this map is a bijection
between the logarithmic connections on $\xi$ and logarithmic connections
on $\xi^{\otimes m}$. Therefore, the above logarithmic connection on
$L^{\otimes 2}$ induces an integrable logarithmic connection
on $L$ singular over $D$.

As before, let ${\mathcal G}_L$ be the
complex group consisting of all pairs of the form $(h\, , \phi)$, where
$h\,\in\, \text{Aut}(X)$ and $\phi$ is a holomorphic automorphism of the
line bundle $L$ over the automorphism $h$ of $X$.
It is easy to see that ${\mathcal G}_L$ is
identified with $\text{GL}(2, {\mathbb C})$. Indeed, the tautological
action of $\text{GL}(2, {\mathbb C})$ on ${\mathcal
O}_{{\mathbb P}(M(2,{\mathbb C}))}(-1)$ produces this isomorphism.

The holomorphic line bundle $L$ does not admit an equivariant structure.
This follows immediately from the fact that there is no nontrivial
homomorphism from $\text{PGL}(2, {\mathbb C})$ to $\text{GL}(2, {\mathbb C})$
(in particular, the projection homomorphism
$$
\text{GL}(2, {\mathbb C})\,\longrightarrow\, \text{PGL}(2, {\mathbb C})
$$
does not split). Therefore, the condition in Theorem \ref{thm:prop-G-1}
that $G$ is simply connected is essential.

The pullback of $L\,=\, {\mathcal O}_{{\mathbb P}(M(2,{\mathbb C}))}(-1)$
by any automorphism of ${\mathbb P}(M(2,{\mathbb C}))$ is holomorphically
isomorphic to $L$. But, as we have seen above, the
holomorphic line bundle $L$ does not admit an equivariant structure.
Therefore, Proposition 1.2.1 of \cite{Kl} is not
valid for a general log parallelizable $G$--pair.

\section{Semistability and restriction to invariant curves} \label{sec:4}

Let $X$ be a complete toric variety defined over an algebraically
closed field $k$. We do not make any assumption on the characteristic
of $k$. By an invariant curve in $X$ we will mean pair
$(C\, ,f)$, where $C$ is an irreducible smooth projective curve defined
over $k$, and $f\, :\, C\,\longrightarrow\, X$ is a separable morphism,
which is an embedding over a nonempty open subset of $C$, such that
$f(C)$ is preserved by the action of the torus on $X$. If $(C\, ,f)$
is an invariant curve, then $C$ is a rational curve.

\begin{proposition}\label{prop1}
Let $E$ be an equivariant vector bundle of rank $r$ over $X$. The
following two statements are equivalent:
\begin{enumerate}
\item For every invariant curve $(C\, ,f)$, the pullback $f^*E$
is semistable.

\item There is a line bundle $L$ over $X$ such that the
vector bundle $E$ is isomorphic to $L^{\oplus r}$.
\end{enumerate}
\end{proposition}

\begin{proof}
The second statement in the proposition evidently implies the first statement.

Assume that the first statement in the proposition holds. Let
$End(E)\, =\, E\otimes E^*\, \longrightarrow\, X$ be the endomorphism
bundle. Take any
invariant curve $(C\, ,f)$. Since $C$ is isomorphic to ${\mathbb P}^1_k$,
the vector bundle $f^*E$ splits into a direct sum of line bundles
\cite{Gr}. Therefore, the given condition that $f^*E$ is semistable implies
that $f^*E$ is isomorphic to $\xi^{\oplus r}$ for some line bundle $\xi$
on $C$.
Hence the vector bundle $f^*End(E)\,=\, End(f^*E)$ is trivial. From this
it follows that the vector bundle $End(E)$ is trivial
\cite[p. 633, Theorem 6.4]{HMP}.

All functions on $X$ are constants. So for any global endomorphism
$\phi$ of $E$, the coefficients of the characteristic polynomial
of $\phi$ are constant functions. Hence the set of eigenvalues of $\phi_x
\, \in\, End(E_x)$ is independent of $x\, \in\, X$.

Let
$$
\phi\,\in\, H^0(X,\, End(E)) \, =\, k^{r^2}
$$
be a section such that the eigenvalues of $\phi$ are distinct.

The eigenspace decomposition of $\phi_x$, $x\,\in\, X$, produces a
decomposition
$$
E\,=\, \bigoplus_{i=1}^r L_i
$$
into a direct sum of line bundles. From this it follows that
\begin{equation}\label{e1}
End(E) \, =\, \bigoplus_{i,j=1}^r L_j\otimes L^*_i\, .
\end{equation}
On the other hand,
\begin{equation}\label{e2}
End(E) \, =\, {\mathcal O}^{\oplus r^2}_X\, .
\end{equation}
Comparing \eqref{e1} and \eqref{e2}, from \cite[p. 315, Theorem 2]{At}
we conclude that $L_j\otimes L^*_i\, =\, {\mathcal O}_X$ for all $i$ and $j$.
In other words, $L_i \,\cong \, L_j$ for all $1 \leq i\, ,j\,\leq r$.
Therefore, the second statement in the proposition holds.
\end{proof}

\begin{remark}
{\rm The isomorphism in statement (2) in Proposition \ref{prop1} is not necessarily
equivariant. To see this, take a line bundle $L\longrightarrow X$
with two (non-isomorphic) different toric actions; for instance, take $X\,=\,{\mathbb P}^1$,
$L\,=\,{\mathbb P}^1\times {\mathbb C}$, $\varphi_1((x,y),t)\,=\,(tx,y)$ and  $\varphi_2((x,y),t)
\,=\, (tx, ty)$. Denote $L_1, L_2$ the two
corresponding equivariant line bundles. Then $E=L_1\oplus L_2$
satisfies the statements in Proposition \ref{prop1}. However,
there is no equivariant line bundle $L'$ such that $E\cong L'\oplus L'$. Compare
with \cite[Corollary~1.2.4]{Kl}.}
\end{remark}

\section{Chern classes of equivariant bundles} \label{sec:5}

Let $X$ be a smooth complete complex toric variety of dimension
$d$ so that the boundary is a simple normal crossing divisor.
In \cite[Theorem 3.2.1]{Kl}, the Chern classes of an
equivariant vector bundle $E\longrightarrow X$ are computed through a resolution.
In this section we will compute the Chern classes of an equivariant vector bundle
$E$ through the natural logarithmic connection on $E$.
As the Newton classes $N_p(E)$ (the sum over
the $p$-th powers of the Chern roots of $E$) can be expressed in terms of the
residues of an integrable logarithmic connection (see \cite[Corollary
B.3]{EsVi}), we use such formula to compute the Chern character in our case.

We begin with introducing some notation. The fan associated to $X$ in $N\,\cong
\,\ZZ^d$ will be denoted by $\Delta$. The complex torus acting on $X$ is
${\mathcal T}\,=\,T_N\,=\,N\otimes \mathbb{C}^\ast$. As before, let ${
\mathcal U}\,\subset\,X$ be the dense open orbit of ${\mathcal T}$, and denote
 $$
 D\,=\, X \setminus {\mathcal U}\, ,
 $$
which is a normal crossing divisor. 

For any cone $\sigma \,\in\, \Delta$, we have an open subset
\begin{equation}\label{eq:open_subset}
{\mathcal U}_\sigma \,\cong \,\mathbb{C}^{k} \times (\mathbb{C}^\ast)^{d-k}\, ,
\end{equation}
where $k\,=\,\dim \sigma$. This open subset has a distinguished point $x_\sigma
\,\in\,
 {\mathcal U}_\sigma$ and its coordinates under the identification~\eqref{eq:open_subset} are
 $x_\sigma = (0,\stackrel{(k)}{\cdots},0,1,\stackrel{(d-k)}{\cdots},1)$. The
 orbit of $x_\sigma$ is denoted by $\mathcal{O}_\sigma = {\mathcal T} \cdot x_\sigma$. This is
 a torus of dimension~$d-k$. The stabilizer of such a point is the subgroup
  $$
  T_\sigma\,=\, \ker (T_N \to \mathcal{O}_\sigma) \,\subset\, T_N\, .
  $$
The closure of $\mathcal{O}_\sigma$ in $X$ will be denoted by $X_\sigma$; it
is also a toric variety. For one dimensional
 cone $\delta\in \Delta_1$, such subvarieties $X_\delta$ correspond to
divisors, and
\[
 D\,=\,\bigcup_{\delta\in \Delta_1} X_\delta\, .
\]

Let $E$ be an equivariant vector bundle of rank $r$ on $X$, and let $\nabla$ be
the integrable logarithmic connection on $E$ constructed in Proposition \ref{prop2}.
For any cone $\sigma\in \Delta$, the equivariant structure is given
locally by a representation $\rho_\sigma\,:\, T_\sigma \,\longrightarrow\,
\text{GL}\,(V)$ which extends to
$T_N$. Indeed, the action of ${\mathcal T}=T_N$ on $E$ produces an action 
of $T_N$ on $E|_{X_\sigma}$, reducing to an action of $T_\sigma$ on the fixed fiber, i.e.,
on $V$. 

The linear representation $\rho_\sigma$ splits into isotypical
components
\begin{equation}\label{eq:iso_comp}
V\,=\,\bigoplus_{\chi\in M} V^\sigma(\chi)\, ,
\end{equation}
where $\rho_\sigma(t)v\,=\,\chi(t)\cdot v$ for $v\,\in\,V^\sigma(\chi)$. Here
$M\,=\,\mathrm{Hom}_\ZZ(N,\ZZ)$ is the dual lattice of $N$.

Fix a $1$-dimensional cone $\delta\,\in\,\Delta$ generated by some $n
\,\in\, N$. As a result of~\eqref{eq:iso_comp} there is a basis $v_1,\ldots,v_r\,\in
\,V$ where $v_i\,\in\, V(\chi_i)$ and $\chi_i\,=\,\chi_{m_i}$, 
$m_i\,\in\, M$,
$i\,=\,1,\ldots, r$. The sections
$v_i(t)\,=\,\rho_\delta(t)\cdot v_i\,\in\, \Gamma({\mathcal U},E)$ satisfy
the condition $\nabla v_i\,=\,0$ and
extend to $\mathcal{U}_\delta={\mathcal U} \cup X_\delta$; here
we are using the alternative description of a logarithmic connection
(see \eqref{co-op}). These extended
sections are denoted by $v_i$.
 The section $s_i(t)\,=\,\chi_i(t)^{-1} v_i(t)\,\in\, \Gamma({\mathcal U} \cup X_\delta,E)$ is
 a non-vanishing one and
 \begin{equation}\label{eq:der_sect}
	 \nabla s_i\,=\,\frac{\text{d} \chi_i^{-1}}{\chi_i^{-1}}\, 
	 s_i\,=\, - \frac{\text{d}\chi_{m_i}}{\chi_{m_i}}\, s_i\, .
 \end{equation}
 gives the following local expression for $\nabla$
 \begin{equation}\label{eq:matrix_connection}
	 \omega_\delta\,=\,- \text{diag}\left(\frac{\text{d}\chi_{m_i}}{\chi_{m_i}}\, ; \,i=1,\ldots,r
	 \right).
 \end{equation}
 Let $\text{Res}_\delta\,:\,E(-\log D)\,\longrightarrow\, E|_{X_\delta}$ be the
Poincar\'e residue map, where
$E(-\log D) = E \otimes TX(-\log D)$, and let
  $$
  \Gamma_\delta\,=\,\text{Res}_\delta\circ\nabla
  $$
 be the function defined on~\cite[Appendix B]{EsVi}. Clearly $\Gamma_\delta(s_i)
\,=\, -\langle m_i,n \rangle
 \left( s_i|_{X_\delta} \right)$.

 Let $\Delta_1=\left\{ \delta_1,\ldots,\delta_s \right\}$ be the set of
 one-dimensional faces where the $i$-th face is generated by a primitive
 element $n_i\in N$. Recall the isotypical decomposition~\eqref{eq:iso_comp} associated
 to $\rho_{\delta_j}$ 
 \[
 V=\bigoplus_{i=1}^r V^j(\chi_{m_{i,j}}).
 \]
 Then,
 \begin{equation}\label{eq:gamma_comp}
	 \Gamma_{\delta_1}^{\alpha_1}\circ
	 \cdots\circ\Gamma_{\delta_s}^{\alpha_s}=(-1)^{\alpha_1+\ldots+\alpha_s}\text{diag}\left(
	 \prod_{j=1}^s \langle
	 m_{i,j},n_{j}\rangle^{\alpha_j}\,:\,i=1,\ldots,r
	 \right),
 \end{equation}
for $\alpha_j\geq 0$, $j=1,\ldots,s$.
 Now we apply the formula in~\cite[Cor. B.3]{EsVi},
 \[
 N_p(E)=(-1)^p\sum_{\alpha_1+\ldots+\alpha_s=p}\frac{p!}{\alpha_1!\cdots\alpha_s!}
 \text{Tr}\left(\Gamma_{\delta_1}^{\alpha_1}\circ
 \cdots\Gamma_{\delta_s}^{\alpha_s}  \right)[X_{\delta_1}]^{\alpha_1}\cdots
 [X_{\delta_s}]^{\alpha_s},
 \]
to compute the
Newton classes $N_p(E)$ of $E$, i.e, the sum of
$p$-th powers of the Chern roots.

 Using~\eqref{eq:gamma_comp} we have
 \begin{align*}
 N_p(E)&=\sum_{\alpha_1+\ldots+\alpha_s=p}\frac{p!}{\alpha_1!\cdots\alpha_s!}
 \left( \sum_{i=1}^r \left( \prod_{j=1}^s \langle
 m_{i,j},n_{j}\rangle^{\alpha_j}
 \right)[X_{\delta_1}]^{\alpha_1}\cdots[X_{\delta_s}]^{\alpha_s} \right)\\
 &=\sum_{i=1}^r \sum_{\alpha_1+\ldots+\alpha_s=p}
 \frac{p!}{\alpha_1!\cdots\alpha_s!} \left( (\langle
 m_{i,1},n_{1}\rangle [X_{\delta_1}])^{\alpha_1}\cdots (\langle
 m_{i,s},n_{s}\rangle [X_{\delta_s}])^{\alpha_s}\right)\\
 &=\sum_{i=1}^r \left( \sum_{j=1}^s \langle
 m_{i,j},n_{j}\rangle[X_j] \right)^p.
 \end{align*}
 Set
  \begin{equation}\label{dmn}
 d_{nm}\, =\, \dim V^{\delta}(\chi_m)\, ,
 \end{equation}
where $\delta=\mathbb{R}_+ \cdot n$. 
Then we enlarge the rank of summation as follows:
  \begin{equation}\label{eq:newton_poly}
	 N_p(E)=\sum_{m\in M} \left( \sum_{n\in \Delta_1 }d_{nm} \langle
	 m,n\rangle[X_{n}] \right)^p.
 \end{equation}
 where $X_n$ is the divisor associated to the one-dimensional face
 $\delta=\mathbb{R}_+\cdot n$.

 This gives the Chern character
 \[
 \text{ch}\,(E)=\sum_i \exp(\xi_i)= \sum_{p=0}^\infty \frac1{p!} N_p(E)\, ,
 \]
 where $\xi_i$ are the Chern roots. Here
 \begin{align*}
 \text{ch}(E)=&\sum_{p=0}^\infty \sum_{m\in M} \frac1{p!}\left( \sum_{n\in
 \Delta_1 } d_{nm}\langle
 m,n\rangle[X_n] \right)^p\\
 =&\sum_{m\in M} \exp\left( \sum_{n\in \Delta_1 } d_{nm}
 \langle m,n\rangle [X_n] \right)\\
 =&\sum_{m\in M} \prod_{n\in \Delta_1 }
 \exp( d_{nm} \langle m,n\rangle [X_n]).
 \end{align*}

 From~\eqref{eq:newton_poly}, one easily compute the Chern classes. For
 $m\,\in\, M$, define 
 $$
 L_m\,:=\,\sum_{n\in \Delta_1} d_{nm} \langle m,n\rangle \, [X_n]\, .
 $$
 Thus, $N_p(E)=\sum_{m\in M} L_m^p$ is 
 the sum of $p$-powers over $M$ which are by definition the Newton
 polynomials (note that this is a finite sum). Hence $L_m$ are the Chern roots.
Now it is straightforward to check that
 \begin{equation}\label{eq:Chern_class}
	 c_k(E)=\sum_{\substack{m_1,\ldots,m_k\in M\\m_i\neq m_j}} \prod_{i=1}^k
	 L_{m_i}=\sum_{\substack{m_1,\ldots,m_k\in M\\m_i\neq m_j}}
	 \prod_{i=1}^k \left(\sum_{n\in \Delta_1}  d_{n m_i} \langle
	 m_i,n\rangle\, [X_n]\right).
 \end{equation}

For $k=1$, the class $c_1(E)$ coincides with the one in \cite{Kl}. This is
done in \cite{Kl}, (3.2.4).

\end{document}